\documentclass[11pt, reqno]{amsart}
\usepackage{amssymb}
\usepackage{stmaryrd}
\usepackage{epsfig}
\usepackage{mathdots}
\usepackage{hyperref}
\hypersetup{colorlinks=true,linkcolor=blue,citecolor=magenta}
\usepackage{fullpage}
\usepackage{float}
\usepackage{color}
\usepackage{diagbox}
\usepackage{slashbox}
\usepackage{tikz}
\usepackage{booktabs}
\usetikzlibrary{tikzmark}
\usepackage{tikz-cd}
\usepackage{pgfplots}
\pgfplotsset{compat=1.18}
\usepackage[OT2,T1]{fontenc}
\DeclareSymbolFont{cyrletters}{OT2}{wncyr}{m}{n}
\DeclareMathSymbol{\Sha}{\mathalpha}{cyrletters}{"58}
\usepackage{array, longtable}
\newcolumntype{C}{>{$}c<{$}}

\usepackage{tikz,colortbl}
\usetikzlibrary{calc}
\usepackage{zref-savepos}

\newcounter{NoTableEntry}
\renewcommand*{\theNoTableEntry}{NTE-\the\value{NoTableEntry}}

\newcommand{\genlegendre}[4]{%
  \genfrac{(}{)}{}{#1}{#3}{#4}%
  \if\relax\detokenize{#2}\relax\else_{\!#2}\fi
}

\theoremstyle{plain}

\newtheorem{theorem}{Theorem}[section]
\newtheorem{corollary}[theorem]{Corollary}
\newtheorem{lemma}[theorem]{Lemma}
\newtheorem{proposition}[theorem]{Proposition}

\theoremstyle{definition}

\newtheorem{remark}[theorem]{Remark}

\numberwithin{equation}{section}

\setbox0=\hbox{$+$}
\newdimen\plusheight
\plusheight=\ht0
\def\+{\;\lower\plusheight\hbox{$+$}\;}

\setbox0=\hbox{$-$}
\newdimen\minusheight
\minusheight=\ht0
\def\-{\;\lower\minusheight\hbox{$-$}\;}

\setbox0=\hbox{$\cdots$}
\newdimen\cdotsheight
\cdotsheight=\plusheight
\def\cds{\lower\cdotsheight\hbox{$\cdots$}}

\def\C{\mathbb{C}}

\begin{document}
\title{Central $L$ values of congruent number elliptic curves}

\author{Xuejun Guo}
\author{Dongxi Ye}
\author{Hongbo Yin}

\address{
\parbox[t]{15cm}{School of Mathematics, Nanjing University, Nanjing 210093, People’s Republic of China}}
\email{guoxj@nju.edu.cn}

\address{
\parbox[t]{15cm}{School of Mathematics (Zhuhai), Sun Yat-sen University, Zhuhai 519082, Guangdong,
People's Republic of China
}
}
\email{yedx3@mail.sysu.edu.cn}

\address{
\parbox[t]{15cm}{
    School of Mathematics, Shandong University, Jinan, Shandong, 250100, People's Republic of China\\
 State Key Laboratory of Cryptography and Digital Economy Security, Shandong University, Jinan, 250100, People's Republic of China
  }%
}
\email{yhb@sdu.edu.cn}



\thanks{The first author was supported by the Natural Science Foundation of China (Grant No. 12231009). The second author was supported by the Guangdong Basic and Applied Basic Research Foundation (Grant No. 2024A1515030222). The third author was supported by the National Key Research and Development Program of China (Grant No.2021YFA1000700) and the Natural Science Foundation of
 Shandong Province (Grant No.ZR2023MA001)}

\begin{abstract}
Let  $E_n$  be the congruent number elliptic curve $y^2=x^3-n^2x$, where $n$ is square-free and not divisible by primes $p\equiv 3\pmod 4$. In this paper, we prove that
$L(E_n,1)$ can be expressed as the square of CM values of some simple theta functions, generalizing two classical formulas of Gauss. Our result is meaningful  in both theory and practical computation. 
\end{abstract}

\maketitle
\allowdisplaybreaks

\section{Introduction}
In October 1798, Gauss proved the following  formulas in his study of the arithmetic-geometric mean,

\begin{align}
\sqrt{\frac{\varpi}{\pi}}& = 1 - 2e^{-\pi} + 2e^{-4\pi} - 2e^{-9\pi} + \cdots\text{,} \label{eq:series1}\\
\sqrt{\frac{\varpi}{\pi}}& =2\left(e^{-\frac{1}{4}\pi} + e^{-\frac{9}{4}\pi} +e^{-\frac{25}{4}\pi} + \cdots\right)\text{,} \label{eq:series2} 
\end{align}
where 
\[
\varpi=2 \int_0^1 \frac{\mathrm{~d} t}{\sqrt{1-t^4}} 
\] is the lemniscate constant (see page 418 of \cite{Gauss1}).
 By computing only  the first three terms of equation (\ref{eq:series2}), Gauss was able to determine $\varpi$ with an accuracy of up to 26 decimal places.
See David A. Cox's paper \cite{Cox1984} for an excellent account of Gauss's work on the arithmetic-geometric mean.

It is not hard to see that the right-hand side of equation (\ref{eq:series1}) is a special value of the Jacobi theta function $\theta(\tau)=\sum_{n\in \mathbb{Z}}e^{2\pi i n^2\tau}$ at $\tau=(i+1)/2$, while the right-hand side of equation (\ref{eq:series2}) is a special value of $\theta_{\chi_4}(\tau)=\sum_{n\in \mathbb{Z}}\chi_4(n)e^{2\pi i n^2\tau}$ at $\tau=i/8$, where $\chi_4$ is the trivial character with conductor $4$.
Let $E_n$ be the elliptic curve defined by $y^2=x^3-n^2x$, and let $L(E_n,s)$ be the Hasse-Weil $L$-function of   $E_n$. 
In modern language, Gauss's formulas are equivalent to the following equations.

\begin{align}
2\sqrt{\frac{L(E_1,1)}{\pi}}&=\theta\left(\frac{i+1}{2}\right)\text{,} \label{eq1.3}\\
2^{\frac{7}{4}}\sqrt{\frac{L(E_2,1)}{\pi}}&=\theta_{\chi_4}\left(\frac{i}{8}\right)\text{.}\label{eq1.4}
\end{align} 
In this paper, we will generalize Gauss's  formulas to all $E_n$ with $n$ square-free and not divisible by primes congruent to $3$ modulo $4$. 

For a positive integer~$n$, assumed to be square-free throughout the remainder of this paper, let
$\chi_{n}(\cdot)=\left(\frac{\cdot}{n}\right)$ be the Kronecker symbol, and let $$
\theta_{\chi_n}(\tau)=\sum_{k\in \mathbb{Z}}\chi_n(k)e^{2\pi i k^2\tau}.
$$

\begin{theorem}\label{Theorem1.1}
Let  $m=\prod_{i=1}^s p_i$ with  each prime $p_i\equiv1\pmod{4}$, $b$  an even integer such that $b^2\equiv -1\pmod{m^2}$, $\tau_m=\frac{b+i}{2
m^2}$, $\tau'_{m}=\frac{b+m^2+i}{2
m^2}$. Then 

    \[
L(E_n,1)=\begin{cases}\displaystyle{\frac{\pi \left\lvert\theta_{\chi_{n}}\left(\tau_n\right)\right\rvert^2}{4\sqrt{2}n }},\ &\text{if }\ n=m;\\[1.5ex]
\displaystyle\frac{\pi \left\lvert\theta_{\chi_{n/2}}(\tau_{n/2}')\right\rvert^2}{\sqrt{2}n }  ,\ &\text{if }\ n=2m .
\end{cases}
\]
In particular, $L(E_n,1)=0$ if and only if $\theta_{\chi_{n}}\left(\tau_n\right)=0$ when $n$ is odd,  or $\theta_{\chi_{n/2}}(\tau'_{n/2})=0$  when $n$ is even.
\end{theorem}


Note that in the case $n=2m$, we also have $\left|\theta_{\chi_{n/2}}(\tau'_{n/2})\right|={\left|\theta_{\chi_{2n}}\left(\tau_{n/2}/4\right)\right|} $ (see Remark~\ref{rem1}). 
If $n=1$, then we can take   $b=0$. By Theorem \ref{Theorem1.1}, we have \[
L(E_1,1)=\frac{\pi \left|\theta\left(\frac{i}{2
}\right)\right|^2}{4\sqrt{2} }
\]
which reduces to  Gauss's  formula by applying 
  the identity $ \theta^2\left(\frac{i}{2}\right)= \sqrt{2}\theta^2\left(\frac{1+i}{2}\right)$. If $n=2$, then $m=1$, and $b=1$. Then Theorem~\ref{Theorem1.1} gives
\[
L(E_2,1)=\frac{\pi \left|\theta_{\chi_4}\left(\frac{1+i}{8
}\right)\right|^2}{2\sqrt{2} }
\]
which reduces to  Gauss's  formula by applying 
  the identity $ \theta_{\chi_4}\left(\frac{i}{8}\right)= \theta_{\chi_4}\left(\frac{1+i}{8}\right)$. 

Writing  central values of $L$-functions as special values of special functions has a long history, see for example \cite{H,S1,S2,R1}. Rodriguez-Villegas \cite{V1,V1.5} first related the central values of Hecke $L$-functions of imaginary quadratic fields to the square of CM values of half weight theta functions. As an application, he obtained an effective way to compute the predicted order of the Tate-Shafarevich group of the Gross elliptic curves. His  proof contains two main ingredients. One is Hecke's formula expressing central $L$-values as a sum of values of binary theta functions, which is in fact a consequence of Kronecker's limit formula. The other is a factorization lemma that expresses
  the CM values of the binary theta functions as the product of CM values of the  theta functions. This is also the method adopted in this paper and many other works \cite{V2,V3,V4} that yield similar results. Using a different method, i.e., the theta lifting for the reductive pair $(U(1),U(1))$ developed in \cite{Y2}, Rodriguez-Villegas and Yang \cite{V4} generalized these works to some special Hecke characters of more general CM fields. In the subsequent work, Yang and his collaborators \cite{Y3,Y4,Y5} further strengthened the results in this direction.

It is worth noting that all the aforementioned studies assume the imaginary quadratic field $K$ to have discriminant $D<-4$. The cases $D=-3,-4$ are very special in CM theory as their ring of integers has more units. This always causes extra subtleties when one investigates relevant topics. For example, the modular curve corresponding to the theta function $\theta_{\chi_n}$
is $X_0(4n^2)$. In order to avoid technical complexity of the computation of theta lifts at dyadic primes, most of the previous research assume  the ``Heegner Hypothesis" which requires the prime factors of $4n^2$ to be split in $K=\mathbb{Q}(\sqrt{-D})$. Hence, the case $D=-4$ is always excluded. However, in her recent beautiful work \cite{R}, Rosu deals with $D=-3$ case and succeeds in expressing the central $L$-values of the cube sum elliptic curve $A_n: x^3+y^3=n$ as the square of CM values of theta functions. This has arithmetic significance, as $A_n$ is closely connected  to the classical Diophantine problem of writing an integer as the sum of two rational cubes. In this paper, we consider the very last case $D=-4$ and deal with the congruent elliptic curve which is closely related to the  famous congruent number problem in arithmetic. What is interesting is that, compared to previous works, the expression in our situation is significantly
  simpler. 


Theorem \ref{Theorem1.1} has several potential applications. Recall that a positive integer is called a congruent number if it is the area of a  right triangle with  rational side lengths.  The determination of which integers are congruent numbers or not originates from an Arabic manuscript in 972 CE and is one of the oldest unsolved arithmetic problems in number theory. It is conjectured that all integers $n\equiv 5,6,7\pmod{8}$ are congruent numbers while $100\%$ integers $n\equiv 1,2,3\mod 8$ are not, see \cite{TYZ}. The second part of the conjecture has been proved by Burungalea and Tian proved in \cite{BT2}.  
For the diverse results on congruent number see \cite{BT2,BT,T, TYZ,Smith,Khanra,Tunnell,Qin}. We just mention that all primes $p\equiv 5,7\mod 8$ are proved to be congruent numbers while all primes $p\equiv 3\mod 8$ are not. But for primes $p\equiv 1 \mod 8$, we still do not know whether there are infinitely many congruent numbers. It is well known that $n$ is a congruent number if and only if $E_n$ has infinitely many rational points if and only if $L(E_n,s)=0$, assuming the Birch and Swinnerton-Dyer (BSD) conjecture is true. Our formulae can be used effectively to predict which integers are congruent numbers and to compute the predicted orders of the Tate-Shafarevich groups of the rank 0 congruent number elliptic curves. Note that by the result in \cite{Tunnell,Qin}, we know that if our theta value is less than a bound, then it should be zero, for details see Proposition \ref{prop8.4}. Beside this, we will show another application in the study of the zeros of theta functions. If the number $n$ in Theorem \ref{Theorem1.1} satisfies $n\equiv 5\pmod 8$, then $L(E_n,1)=0$ due to the fact that the root number of $E_n$ is $-1$. By Theorem \ref{Theorem1.1}, this implies that $\theta_{\chi_n}(\tau_n)=0$. Furthermore, if  $n\equiv 1\pmod 8$ and  $L(E_n,1)=0$, then the order of vanishing  
of  $\theta_{\chi_n}(\tau)$  at $\tau_n$ is at least $2$.

The remainder of this paper is organized as follows. In Sections~\ref{charchi} and~\ref{classfield}, we review  some basic properties of the Hecke character of $E_{n}$ and its associated class fields. Upon these, we decompose $L(E_{n},s)$ in terms of CM values of non-holomorphic Eisenstein series and consequently  express $L(E_{n},1)$ as a linear combination of CM values of the  theta function in Section~\ref{jacobitheta}. As a byproduct, we can indeed express $L(E_{n},1)$ as a Galois trace of some CM value of the  theta function for any $n$. In Section~\ref{factorlem}, we review a factorization formula, which, according to Rodriguez-Villegas \cite{V1}, in fact originates from a somewhat forgotten result of Kronecker.
 Making use of the factorization formula we are able to ``factorize'' the CM values of the theta function in terms of CM values of some particular unary theta functions which will be a key ingredient in our proof of Theorem~\ref{Theorem1.1}. After these, the proof of Theorem~\ref{Theorem1.1} will be established in Section~\ref{square}. Finally, in Section~\ref{mock} we conclude  with a discussion on potential applications of Theorem~\ref{Theorem1.1} to finding mock Heegner zeros of some Dirichlet theta functions.

{\bf Acknowledgment.} The authors are grateful to Tonghai Yang for many helpful and insightful discussions related to this work. The first author also thanks the University of Wisconsin–Madison for its hospitality and support during the completion of this project.
 



\section{The Hecke character $\tilde{\chi}_{n}$}\label{charchi}

Let $n$ be a positive integer such that $n=2^{e}m$ with $e=0$ or $1$, and $m$ being odd. Fix $K=\mathbb{Q}(i)$, $\mathcal{O}_{K}=\mathbb{Z}[i]$, and $$
\mathfrak{m}=\begin{cases}
    (2+2i)n\mathcal{O}_{K},&\mbox{if $n$ is odd;}\\
    2n\mathcal{O}_{K},&\mbox{otherwise.}
\end{cases}
$$
Denote by $J_{\mathfrak{m}}^{*}$ the multiplicative group generated by the integral ideals of $\mathcal{O}_{K}$ coprime to $\mathfrak{m}$, and denote by $J_{\mathfrak{m},\mathbb{Z}}^{*}$ its subgroup generated by the principal integral ideals $\alpha\mathcal{O}_{K}$ with 
$$
\alpha\equiv \begin{cases}
    1\pmod{(2+2i)\mathcal{O}_{K}}\\
    \ell\pmod{n\mathcal{O}_{K}}
\end{cases}\quad\mbox{if $n$ is odd,}
$$
or
$$
\alpha\equiv \begin{cases}
    1\pmod{4\mathcal{O}_{K}}\\
    \ell\pmod{m\mathcal{O}_{K}}
\end{cases}\quad\mbox{if $n=2m$,}
$$
for some integer~$\ell$ coprime to~$n$. Note that one has the identification
$$
J_{\mathfrak{m}}^{*}/J_{\mathfrak{m},\mathbb{Z}}^{*}\cong\begin{cases}
(\mathbb{Z}/n\mathbb{Z})^{\times},&\mbox{if $n$ is odd;}\\
(\mathbb{Z}/2n\mathbb{Z})^{\times}=\mathbb{Z}/2\mathbb{Z}\times(\mathbb{Z}/m\mathbb{Z})^{\times},&\mbox{if $n=2m$ with $m$ odd,}
\end{cases}
$$
and the following technical lemma, whose proof is left to the reader.

\begin{lemma}\label{lemalpha}
\begin{enumerate}
    \item When $n$ is odd,  for a class $[\mathcal{A}]\in J_{\mathfrak{m}}^{*}/J_{\mathfrak{m},\mathbb{Z}}^{*}$, one can always pick a representative $\mathcal{A}=\alpha_{\mathcal{A}}\mathcal{O}_{K}$ with $\alpha_{\mathcal{A}}\equiv1\pmod{4\mathcal{O}_{K}}$.

   \item  When $n=2m$ with $m$ being odd, for a class $[\mathcal{A}]\in J_{\mathfrak{m}}^{*}/J_{\mathfrak{m},\mathbb{Z}}^{*}$ such that  $\mathcal{A}=\alpha_{\mathcal{A}}\mathcal{O}_{K}$ with $\alpha_{\mathcal{A}}\equiv1\pmod{4\mathcal{O}_{K}}$, there is different class $[\mathcal{B}]$ from $[\mathcal{A}]$ such that $\mathcal{B}=\alpha_{\mathcal{B}}\mathcal{O}_{K}$ with $\alpha_{\mathcal{B}}\equiv3+2i\pmod{4\mathcal{O}_{K}}$ and $\alpha_{\mathcal{B}}\equiv \alpha_{\mathcal{A}}\pmod{m\mathcal{O}_{K}}$. So if one fixes a 
    $$
    \delta\equiv\begin{cases}
        3+2i\pmod{4\mathcal{O}_{K}}\\
        1\pmod{m\mathcal{O}_{K}},
    \end{cases}
    $$
    then $J_{\mathfrak{m}}^{*}/J_{\mathfrak{m},\mathbb{Z}}^{*}=\{[\mathcal{A}],[\delta\mathcal{A}]:\, \alpha_{\mathcal{A}}\equiv1\pmod{4\mathcal{O}_{K}}\}$.
 \end{enumerate}
   
\end{lemma}



Define for  $x$  such that $i^{j}x\equiv1\pmod{(2+2i)\mathcal{O}_{K}}$, 
$$
\chi_{1}'(x)=i^{j}.
$$
Then for   $x$ coprime to $\mathfrak{m}$, define
$$
\tilde{\chi}_{n}(x)=\chi_{1}'(x)\left(\frac{n}{N(x)}\right)x,
$$
and write $\chi_{n}'(x)=\chi_{1}'(x)\left(\frac{N(x)}{n}\right)$.
Moreover,  it is defined to be $0$ for all $x$      not
 coprime to~$\mathfrak{m}$. 

Note that if $x=q\equiv3\pmod{4}$ is a prime, then 
$$
\tilde{\chi}_{n}(x)=-x\equiv 1\pmod{2+2i}.
$$
So  in terms of ideal, 
$$
\tilde{\chi}_{n}(x\mathcal{O}_{K})=\left(\frac{n}{N(x)}\right)x,\quad\mbox{and}\quad \chi_{n}'(x\mathcal{O}_{K})=\left(\frac{n}{N(x)}\right)
$$
for a generator~$x\equiv1\pmod{2+2i}$.  
If $x=a+bi$ with $N(x)=q\equiv1\pmod{4}$ a prime and $i^{j}x\equiv1\pmod{2+2i}$, then
$$
\tilde{\chi}_{n}(x)=i^{j}\left(\frac{n}{N(x)}\right)x=\left(\frac{n}{N(i^{j}x)}\right)(i^{j}x)
$$
with $i^{j}x\equiv1\pmod{2+2i}$.
So similarly, in terms of ideal
$$
\tilde{\chi}_{n}(x\mathcal{O}_{K})=\left(\frac{n}{N(x)}\right)x,\quad\mbox{and}\quad \chi_{n}'(x\mathcal{O}_{K})=\left(\frac{n}{N(x)}\right)
$$
for a generator~$x\equiv1\pmod{2+2i}$. One therefore can view and extend ${\chi}'_{n}$ as  a character on $J_{\mathfrak{m}}^{*}$ that is clearly trivial on $J_{\mathfrak{m},\mathbb{Z}}^{*}$. 

Finally, it is classically known, e.g., \cite{K}, that for a congruent number elliptic
 curve~$E_{n}$, one has
$$
L(E_{n},s)=\frac{1}{4}\sum_{\alpha\in \mathcal{O}_{K}}\frac{\tilde{\chi}_{n}(\alpha)}{N(\alpha)^{s}}
$$
for ${\rm Re}(s)>1$.

\section{The class fields $H_{\mathfrak{m}}$ and $H_{4n}$}\label{classfield}

Let $n$ be a positive integer such that $n=2^{e}m$ with $e=0$ or $1$, and  $m$  odd. Recall that   $\mathfrak{m}$ was defined  at the beginning of Section~\ref{charchi}. Denote by $J_{4n,\mathbb{Z}}^{*}$ the subgroup of $J_{\mathfrak{m},\mathbb{Z}}^{*}$ generated by principal integral ideals $\alpha\mathcal{O}_{K}$ with $\alpha\equiv \ell\pmod{4n\mathcal{O}_{K}}$ for some integer~$\ell$ coprime to~$4n$. It is clear that $J_{4n,\mathbb{Z}}^{*}\subset J_{\mathfrak{m},\mathbb{Z}}^{*}$ is of index~$2$.

By class field theory, let $H_{\mathfrak{m}}$ and $H_{4n}$ be the class fields over $K$ respectively corresponding  to $J_{\mathfrak{m},\mathbb{Z}}^{*}$ and $J_{4n,\mathbb{Z}}^{*}$. In particular, the class field $H_{4n}$ is known as the ring class field of conductor~$4n$ over~$K$ that is a quadratic extension of $H_{\mathfrak{m}}$. The following lemma specifies this quadratic extension.

\begin{lemma}\label{HH2}
Let $H_{\mathfrak{m}}$ and $H_{4n}$ be as defined  above. Then one has
$$
H_{4n}=H_{\mathfrak{m}}(\sqrt{2}).
$$



\end{lemma}
  
\begin{proof}
The field $H_{4n}$ is a quadratic extension of the field $H_{\mathfrak{m}}$ which is ramified only at dyadic primes. 

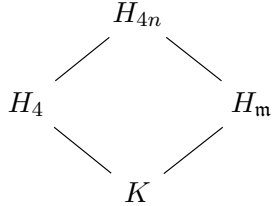
\[
\begin{tikzcd}[row sep=1em, column sep=1.5em, every arrow/.append style = {-}]
 & H_{4n} \arrow[dl] \arrow[dr] & \\[6pt]
H_{4} \arrow[dr] & & H_{\mathfrak{m}} \arrow[dl] \\[6pt]
 & K &
\end{tikzcd}
\]
Since $K(\sqrt{2})$ is the unique quadratic extension of $K$ which is only ramified at $2$, we know that  $H_{4}=K(\sqrt{2})$ which implies that  $H_{4n}=H_{\mathfrak{m}}(\sqrt{2})$. 

\end{proof}


Next note  that $\chi'_{n}$ is trivial on $J_{\mathfrak{m},\mathbb{Z}}^{*}$, and thus, defines a character on both ${\rm Gal}(H_{\mathfrak{m}}/K)$ and ${\rm Gal}(H_{4n}/K)$. So, $\langle\chi'_{n}\rangle=\{{\rm id}, \chi'_{n}\}$ is an order-2 subgroup of the character group of ${\rm Gal}(H_{4n}/K)$. 

Let $L'\subset {\rm Gal}(H_{4n}/K)$ be the stabilized subgroup of $\langle\chi'_{n}\rangle$, and let $L$ be its fixed field. Let $J_L$ be the kernel of $\chi'_n:J_{4n}^* \longrightarrow \mathbb{Z}/2\mathbb{Z}$.  Then $J_L\supset J_{2n,\mathbb{Z}}^*$. Since $\langle\chi'_{n}\rangle$ is trivial on $L'={\rm Gal}(H_{2n}/L)$,  this defines the character group of 
$$
{\rm Gal}(H_{2n}/K)/{\rm Gal}(H_{2n}/L)\cong {\rm Gal}(L/K)\cong \mathbb{Z}/2\mathbb{Z}.
$$
The following lemma is technical but will be useful in our formulation of $L(E_{n},1)$ as a trace.
\begin{lemma}\label{lemchip}
With notation as above, we have $L=K(\sqrt{n})$.  Moreover,  for a class $[\mathcal{A}]\in J_{4n}^{*}/J_{2n,\mathbb{Z}}^{*}$, 
    $$
    \mbox{$\chi'_{n}([\mathcal{A}])=-1$  if and only if $(\sqrt{n})^{\sigma_{\mathcal{A}}}=-\sqrt{n}$.}
    $$
\end{lemma}

\begin{proof}
    By the Artin reciprocity, there
is a subgroup  $G$ of $J_{2n}^{*}$ such that 
$J_{2n}^{*}/G \simeq {\rm Gal}(K(\sqrt{n})/K).
$
By the proof of Theorem 9.18 of \cite{Cox}, $K(\sqrt{n})\subset H_{2n}\subset H_{4n}$.  Note that 
${\rm Gal}(H_{2n,\mathbb{Z}}/K)$
 is a cyclic group. Hence $L=K(\sqrt{n})$ is the unique quadratic sub-extension
  of $K\subset H_{2n}$. 
\end{proof}

    

Finally, we end the present section with a lemma describing
 some special CM values of a theta function that lie in the class field $H_{4n}$. Let     $
    \Theta(\tau)=\sum\limits_{j,k=-\infty}^{\infty}q^{j^{2}+k^{2}},
    $
   which is the square of  the classical  Jacobi theta function. 

\begin{lemma}\label{thetagalois}
We have
    $$
\frac{\Theta(ni/2)}{\Theta(i/2)}\in H_{2n},
$$
and for any $\sigma_{\mathcal{A}}\in {\rm Gal}(H_{2n}/K)$,
$$
\left(\frac{\Theta(ni/2)}{\Theta(i/2)}\right)^{\sigma_{\mathcal{A}}^{-1}}=\frac{\Theta(n\tau_{\mathcal{A}}/2)}{\Theta(\tau_{\mathcal{A}}/2)}.
$$
\end{lemma}

\begin{proof}
    It is clear that $\frac{\Theta(n\tau)}{\Theta(\tau)}$ is a modular function over~$\mathbb{Q}$ for $\Gamma_{0}(4n)$, so it is a rational function over~$\mathbb{Q}$ in $j(4n\tau)$ and $j(\tau)$, and thus, $\frac{\Theta(n\tau/2)}{\Theta(\tau/2)}$ is a rational function over~$\mathbb{Q}$ in $j(2n\tau)$ and $j(\tau/2)$. It is well known that $j(2ni),j(i/2)\in H_{2n}$, and $j(2ni)^{\sigma_{\mathcal{A}}}=j(2n\tau_{\mathcal{A}})$. The desired assertions follow.
\end{proof}

\section{Eisenstein series and the  theta functions}\label{jacobitheta}

Using the definition and properties of the character~$\chi_{n}'$ given in Section~\ref{charchi}, it is straightforward to see that
$$
L(E_{n},s)=\sum_{\substack{\alpha\in\mathcal{O}_{K}\\\alpha\equiv1\pmod{(2+2i)\mathcal{O}_{K}}}}\frac{\chi_n'(\overline{\alpha})}{\alpha|\alpha|^{2s-2}},
$$
where we recall that
$$
\chi'_{n}(\alpha)=
\begin{cases}
    \left(\frac{n}{N(\alpha)}\right),& \mbox{if $(\alpha,\mathfrak{m})=1$;}\\
    0,&\mbox{otherwise.}
\end{cases}
 $$
Let $[\mathcal{A}]$ be class of $J_{\mathfrak{m}}^{*}/J_{\mathfrak{m},\mathbb{Z}}^{*}$ with $\mathcal{A}=\alpha_{\mathcal{A}}\mathcal{O}_{K}=[a,{b+i}{}]$ being primitive with $$\alpha_{\mathcal{A}}\equiv 1\pmod{(2+2i)\mathcal{O_{K}}},$$  
and 
$$
b\equiv \begin{cases}
    1\pmod{2},&\mbox{if $n$ is odd;}\\
    0\pmod{2},&\mbox{otherwise.}
\end{cases}
$$
 Then it is not hard to see that
 \begin{enumerate}
     \item for $n$ being odd, 
     $$
     L(E_{n},s)=\sum_{[\mathcal{A}]\in J_{\mathfrak{m}}^{*}/J_{\mathfrak{m},\mathbb{Z}}^{*}}\chi_{n}'(\alpha_{\mathcal{A}})\alpha_{\mathcal{A}}|\alpha_{\mathcal{A}}|^{2s-2}
  \displaystyle{  \sum_{\substack{\alpha\in \alpha_{\mathcal{A}}\mathcal{O}_{K}\\\alpha\equiv1\pmod{(2+2i)}\\\alpha\equiv\ell\pmod{n}}}\frac{1}{\alpha|\alpha|^{2s-2}}},
     $$

     \item  for $n$ being even, by Lemma~\ref{lemalpha},
     \begin{align*}
          L(E_{n},s)&=\sum_{\substack{[\mathcal{A}]\in J_{\mathfrak{m}}^{*}/J_{\mathfrak{m},\mathbb{Z}}^{*}\\\alpha_{\mathcal{A}}\equiv1\pmod{4\mathcal{O}_{K}}}}\chi_{n}'(\alpha_{\mathcal{A}})\alpha_{\mathcal{A}}|\alpha_{\mathcal{A}}|^{2s-2}
   \displaystyle{ \sum_{\substack{\alpha\in \alpha_{\mathcal{A}}\mathcal{O}_{K}\\\alpha\equiv1\pmod{4}\\\alpha\equiv\ell\pmod{m}}}\frac{1}{\alpha|\alpha|^{2s-2}}}\\
   &\quad -\sum_{\substack{[\mathcal{A}]\in J_{\mathfrak{m}}^{*}/J_{\mathfrak{m},\mathbb{Z}}^{*}\\\alpha_{\mathcal{A}}\equiv1\pmod{4\mathcal{O}_{K}}}}\chi_{n}'(\alpha_{\mathcal{A}})\alpha_{\mathcal{A}}|\alpha_{\mathcal{A}}|^{2s-2}
   \displaystyle{ \sum_{\substack{\alpha\in \alpha_{\mathcal{A}}\mathcal{O}_{K}\\\alpha\equiv3+2i\pmod{4}\\\alpha\equiv\ell\pmod{m}}}\frac{1}{\alpha|\alpha|^{2s-2}}}
     \end{align*}
    
 \end{enumerate}
where $\ell\in\mathbb{Z}$ coprime to~$n$.
Write $\alpha=ja+k(b+i)$, and note that $a\equiv1\pmod{4}$ and $a=N(\alpha_{\mathcal{A}})$ coprime to~$n$. It is routine to check that 
\begin{enumerate}
    \item when $n$ and $b$ are both odd, the conditions on $\alpha$ force $j$ and $k$ to satisfy that
    \begin{align*}
    k&\equiv0\pmod{2n},\quad j\equiv 1\pmod{4},\quad (j,n')=1,
\end{align*}

\item when $n=2m$ with $m$ being odd, $b$ is even, and
$$
\alpha\equiv\begin{cases}
    1\pmod{4},\\
    \ell\pmod{m},
\end{cases}
$$
these amount to have 
 \begin{align*}
    k&\equiv0\pmod{4m},\quad j\equiv 1\pmod{4},\quad (j,m)=1,
\end{align*}

\item when $n=2m$ with $m\equiv\pmod{4}$, $b$ is even, and
$$
\alpha\equiv\begin{cases}
    3+2i\pmod{4},\\
    \ell\pmod{m},
\end{cases}
$$
these amount to have
 \begin{align*}
    k&\equiv2\pmod{4},\quad k\equiv0\pmod{m},\quad j\equiv 3\pmod{4},\quad (j,m)=1.
\end{align*}
\end{enumerate}

Then writing $L(E_{n},s)$ in terms of lattice points, one respectively has
\begin{enumerate}
    \item  when $n$ is odd,
\begin{align}
    &L(E_{n},s)=\sum_{\substack{[\mathcal{A}]\in J_{\mathfrak{m}}^{*}/J_{\mathfrak{m},\mathbb{Z}}^{*}\\\alpha_{\mathcal{A}}\equiv1\pmod{4\mathcal{O}_{K}}\\b\,odd}}\frac{\chi'_{n}(\alpha_{\mathcal{A}})\alpha_{\mathcal{A}}|\alpha_{\mathcal{A}}|^{2s-2}}{N(\mathcal{A})|N(\mathcal{A})|^{2s-2}}\sum_{\substack{k\equiv0\pmod{2}\\j\equiv1\pmod{4}\\(j,n')=1}}\frac{1}{(j+k(n'\tau_{\mathcal{A}}))|j+k(n'\tau_{\mathcal{A}})|^{2s-2}},\nonumber
\end{align}

\item when $n=2m$ with $m$ being odd,
\begin{align}
    &L(E_{n},s)\nonumber\\
     &=\sum_{\substack{[\mathcal{A}]\in J_{\mathfrak{m}}^{*}/J_{\mathfrak{m},\mathbb{Z}}^{*}\\\alpha_{\mathcal{A}}\equiv1\pmod{4\mathcal{O}_{K}}\\b\, even}}\left(\frac{\chi'_{n}(\alpha_{\mathcal{A}})\alpha_{\mathcal{A}}|\alpha_{\mathcal{A}}|^{2s-2}}{N(\mathcal{A})|N(\mathcal{A})|^{2s-2}}\nonumber\sum_{\substack{k\equiv0\pmod{4}\\j\equiv1\pmod{4}\\(j,m)=1}}\frac{1}{(j+k(m\tau_{\mathcal{A}}))|j+k(m\tau_{\mathcal{A}})|^{2s-2}}\right.\nonumber\\
    &\quad\quad\quad\quad\quad\quad\quad\quad\quad\quad\quad-\left.\frac{\chi'_{n}(\alpha_{\mathcal{A}})\alpha_{\mathcal{A}}|\alpha_{\mathcal{A}}|^{2s-2}}{N(\mathcal{A})|N(\mathcal{A})|^{2s-2}}\nonumber\sum_{\substack{k\equiv2\pmod{4}\\j\equiv3\pmod{4}\\(j,m)=1}}\frac{1}{(j+k(m\tau_{\mathcal{B}}))|j+k(m\tau_{\mathcal{B}})|^{2s-2}}\right)\\
    &=\sum_{\substack{[\mathcal{A}]\in J_{\mathfrak{m}}^{*}/J_{\mathfrak{m},\mathbb{Z}}^{*}\\\alpha_{\mathcal{A}}\equiv1\pmod{4\mathcal{O}_{K}}\\b\, even}} \frac{\chi'_{n}(\alpha_{\mathcal{A}})\alpha_{\mathcal{A}}|\alpha_{\mathcal{A}}|^{2s-2}}{N(\mathcal{A})|N(\mathcal{A})|^{2s-2}}\nonumber\sum_{\substack{k\equiv0\pmod{2}\\j\equiv1\pmod{4}\\(j,m)=1}}\frac{1}{(j+k(m\tau_{\mathcal{A}}))|j+k(m\tau_{\mathcal{A}})|^{2s-2}}.
\end{align}
\end{enumerate}
Therefore, writing
$$
n'=\begin{cases}
    n,&\mbox{if $n$ is odd;}\\
    m,&\mbox{if $n=2m$ with $m\equiv 1\pmod{4}$,}
\end{cases}
$$
one can collectively display $L(E_{n},s)$ as
\begin{align}\label{Ls1}
   L(E_{n},s) =\sum_{\substack{[\mathcal{A}]\in J_{\mathfrak{m}}^{*}/J_{\mathfrak{m},\mathbb{Z}}^{*}\\\alpha_{\mathcal{A}}\equiv1\pmod{4\mathcal{O}_{K}}}} \frac{\chi'_{n}(\alpha_{\mathcal{A}})\alpha_{\mathcal{A}}|\alpha_{\mathcal{A}}|^{2s-2}}{N(\mathcal{A})|N(\mathcal{A})|^{2s-2}}\sum_{\substack{k\equiv0\pmod{2}\\j\equiv1\pmod{4}\\(j,n')=1}}\frac{1}{(j+k(n'\tau_{\mathcal{A}}))|j+k(n'\tau_{\mathcal{A}})|^{2s-2}}.
\end{align}
where  the representatives $\mathcal{A}=[a,b+i]$ are taken to be with $b$ being odd for $n$ odd, and even for $n=2m$. 

Upon this, we aim to express the double sums inside the parentheses above in terms of theta functions. To this end, we have to invoke the following two lemmas.

\begin{lemma}\label{eisenstein1}
Write 
$$
n'=\begin{cases}
    n,&\mbox{if $n$ is odd};\\
    m,&\mbox{if $n=2m$.}
\end{cases}
$$
Denote by ${\rm rad}(n')$ the radical of~$n'$. Then one has
    \begin{align*}
        &\sum_{\substack{k\equiv0\pmod{2}\\j\equiv1\pmod{4}\\(j,n')=1}}\frac{1}{(j+k(n'\tau_{\mathcal{A}}))|j+k(n'\tau_{\mathcal{A}})|^{2s-2}}\\
        &=\sum_{\substack{p_{1}\cdots p_{\ell}|{\rm rad}(n')\\ p_{1}\cdots p_{\ell}\equiv r\pmod{4}}}\frac{(-1)^{\ell+\frac{r-1}{2}}}{(p_{1}\cdots p_{\ell})|p_{1}\cdots p_{\ell}|^{2s-2}}\sum_{\substack{k\equiv0\pmod{4}\\j\equiv1\pmod{4}}}\frac{1}{\left(j+k\left(\frac{n'}{p_{1}\cdots p_{\ell}}\tau_{\mathcal{A}}\right)\right)\left|j+k\left(\frac{n'}{p_{1}\cdots p_{\ell}}\tau_{\mathcal{A}}\right)\right|^{2s-2}},
    \end{align*}
    where $\ell$ is set to be~$0$ when $p_{1}\cdots p_{\ell}=1$.
\end{lemma}

\begin{proof}
    This follows immediately from the inclusion-exclusion principle.
\end{proof}


\begin{lemma}\label{thetatheta}
   We have 
    \begin{align*} 
    \underset{s=1}{{\rm CT}}\sum_{\substack{k\equiv0\pmod{2}\\j\equiv1\pmod{4}}}\frac{1}{(j+k\tau)|j+k\tau|^{2s-2}}&=\frac{\pi}{4}\Theta(\tau/2).
    \end{align*}

\end{lemma}

\begin{proof}
    This basically follows from the Hecke trick \cite{H} and the Fourier expansions of classical Eisenstein series of weight~$1$ (see, e.g., \cite{DS}), as well as the classical identity \cite{J}
    $$
    \Theta(\tau)=1+4\sum_{k=1}^{\infty}\left(\sum_{d|k}\chi_{-4}(d)\right)q^{n}.
    $$
    So we omit the details.
\end{proof}

Upon Lemmas~\ref{eisenstein1} and~\ref{thetatheta}, one immediately has 

\begin{corollary}\label{eisentheta}
Write 
$$
n'=\begin{cases}
    n,&\mbox{if $n$ is odd};\\
    m,&\mbox{if $n=2m$}.
\end{cases}
$$
    Denote by ${\rm rad}(n')$ the radical of~$n'$. Then one has
    \begin{align*}
        \underset{s=1}{{\rm CT}} \sum_{\substack{k\equiv0\pmod{2}\\j\equiv1\pmod{4}\\(j,n')=1}}\frac{1}{(j+k(n'\tau))|j+k\tau|^{2s-2}}&=\frac{\pi}{4} \sum_{\substack{p_{1}\cdots p_{\ell}|{\rm rad}(n')\\ p_{1}\cdots p_{\ell}\equiv r\pmod{4}}}\frac{(-1)^{\ell+\frac{r-1}{2}}}{p_{1}\cdots p_{\ell}}\Theta\left(\frac{n'}{p_{1}\cdots p_{\ell}}(\tau/2)\right).
    \end{align*}
\end{corollary}

Combining the identities given in Corollary~\ref{eisentheta} and~\eqref{Ls1} gives
\begin{proposition}\label{main1}
      Denote by ${\rm rad}(n')$ the radical of~$n'$. Then
    $$
    L(E_{n},1)=\frac{\pi}{4}\sum_{\substack{[\mathcal{A}]\in J_{\mathfrak{m}}^{*}/J_{\mathfrak{m},\mathbb{Z}}^{*}\\\alpha_{\mathcal{A}}\equiv1\pmod{4\mathcal{O}_{K}}}}\frac{\chi'_{n}(\alpha_{\mathcal{A}})\alpha_{\mathcal{A}}}{N(\mathcal{A})}\sum_{\substack{p_{1}\cdots p_{\ell}|{\rm rad}(n')\\ p_{1}\cdots p_{\ell}\equiv r\pmod{4}}}\frac{(-1)^{\ell+\frac{r-1}{2}}}{p_{1}\cdots p_{\ell}}\Theta\left(\frac{n'}{p_{1}\cdots p_{\ell}}(\tau_{\mathcal{A}}/2)\right),
    $$
    where $\mathcal{A}=[a,b+i]$ is picked to be a primitive integral ideal
    with $b$ being odd for $n$ odd, and even for $n=2m$.

\end{proposition}
  
As a byproduct of Proposition~\ref{main1}, one can actually formulate $L(E_{n},1)$ as a Galois trace for any $n=2^{e}m$ with $e=0$ or~$1$, and $m$ being odd and square-free. To this end, one needs two additional lemmas. 

\begin{lemma}\label{lemtheta}
     For $\mathcal{A}=\alpha_{\mathcal{A}}=[a,b+i]$ primitive with $\alpha_{\mathcal{A}}\equiv1\pmod{4\mathcal{O}_{K}}$, writing $\tau_{\mathcal{A}}=\frac{b+i}{a}$, one has
        $$
        \frac{\alpha_{\mathcal{A}}}{N(\alpha_{\mathcal{A}})}\Theta(\tau_{\mathcal{A}}/2)=\Theta\left(\frac{1-(-1)^{b}}{4}+\frac{i}{2}\right).
        $$

\end{lemma}

\begin{proof}
We begin by noting that $\alpha_{\mathcal{A}}=ja+k(b+i)$ with $k\equiv0\pmod{4}$ and $j\equiv1\pmod{4}$ and $(j,k)=1$. There are integers~$A$ and~$B$ such that $Aj-8kB=1$, so that $\gamma=\begin{pmatrix}
    A&B\\8k&j
\end{pmatrix}\in\Gamma_{1}(4)$. Then one can deduce by the modularity of $\Theta(\tau)$ that
$$
\Theta(\gamma(\tau_{\mathcal{A}}/2))=(4k\tau_{\mathcal{A}}+j)\Theta(\tau_{\mathcal{A}}/2)=\frac{\alpha_{\mathcal{A}}}{N(\alpha_{\mathcal{A}})}.
$$
Also, it is routine to check by assumption of $b\equiv1\pmod{2}$ that
$$
\gamma(\tau_{\mathcal{A}}/2)=\frac{1}{2}\frac{A\tau_{\mathcal{A}}+2B}{4k\tau_{\mathcal{A}}+j}\equiv\frac{1-(-1)^{b}}{4}+\frac{i}{2}\pmod{\mathbb{Z}}.
$$
This completes the proof.
\end{proof}

\begin{lemma}\label{thetathetagalois}
    Let $[\mathcal{A}]$ be a class  of $J_{\mathfrak{m}}^{*}/J_{\mathfrak{m},\mathbb{Z}}^{*}$ with $\mathcal{A}=\alpha_{\mathcal{A}}=[a,b+i]$ being primitive such that $\alpha_{\mathcal{A}}\equiv1\pmod{4}$ and $b$ is odd, and denote by $\sigma_{\mathcal{A}}$ the Galois element attached to $[\mathcal{A}]$. Then
    \begin{align*}
         &\left(\sum_{\substack{p_{1}\cdots p_{\ell}|{\rm rad}(n')\\ p_{1}\cdots p_{\ell}\equiv r\pmod{4}}}\frac{(-1)^{\ell+\frac{r-1}{2}}}{p_{1}\cdots p_{\ell}}\frac{\Theta\left(\frac{n'}{p_{1}\cdots p_{\ell}}(1+i)/2\right)}{\Theta\left((1+i)/2\right)}\right)^{\sigma_{\mathcal{A}}^{-1}}
         =\sum_{\substack{p_{1}\cdots p_{\ell}|{\rm rad}(n')\\ p_{1}\cdots p_{\ell}\equiv r\pmod{4}}}\frac{(-1)^{\ell+\frac{r-1}{2}}}{p_{1}\cdots p_{\ell}}\frac{\Theta\left(\frac{n'}{p_{1}\cdots p_{\ell}}\tau_{\mathcal{A}}/2\right)}{\Theta\left(\tau_{\mathcal{A}}/2\right)}.
    \end{align*}

   Similarly,  let $[\mathcal{A}]$ be a class  of $J_{\mathfrak{m}}^{*}/J_{\mathfrak{m},\mathbb{Z}}^{*}$ with $\mathcal{A}=\alpha_{\mathcal{A}}=[a,b+i]$ being primitive such that $\alpha_{\mathcal{A}}\equiv1\pmod{4}$ and $b$ is even, and denote by $\sigma_{\mathcal{A}}$ the Galois element attached to $[\mathcal{A}]$. Then
    \begin{align*}
         &\left(\sum_{\substack{p_{1}\cdots p_{\ell}|{\rm rad}(n')\\ p_{1}\cdots p_{\ell}\equiv r\pmod{4}}}\frac{(-1)^{\ell+\frac{r-1}{2}}}{p_{1}\cdots p_{\ell}}\frac{\Theta\left(\frac{n'}{p_{1}\cdots p_{\ell}}(i/2)\right)}{\Theta\left(i/2\right)}\right)^{\sigma_{\mathcal{A}}^{-1}}
         =\sum_{\substack{p_{1}\cdots p_{\ell}|{\rm rad}(n')\\ p_{1}\cdots p_{\ell}\equiv r\pmod{4}}}\frac{(-1)^{\ell+\frac{r-1}{2}}}{p_{1}\cdots p_{\ell}}\frac{\Theta\left(\frac{n'}{p_{1}\cdots p_{\ell}}\tau_{\mathcal{A}}/2\right)}{\Theta\left(\tau_{\mathcal{A}}/2\right)}.
    \end{align*} 
    
\end{lemma}

\begin{proof}
These follow from Lemma~\ref{thetagalois} together with the Shimura reciprocity law.

\end{proof}

  


 Now we are ready to relate $L(E_{n},1)$ to a Galois trace of a theta CM value.

\begin{theorem}\label{main2}
Write 
$$
n'=\begin{cases}
    n,&\mbox{if $n$ is odd};\\
    m,&\mbox{if $n=2m$.}
\end{cases}
$$
and denote by ${\rm rad}(n')$ the radical of~$n'$. Then
$$
L(E_{n},1)=\begin{cases}
    \displaystyle{\frac{\pi}{4}\Theta((1+i)/2)\frac{1}{\sqrt{n}}{\rm Tr}_{H_{\mathfrak{m}}/K}(\sqrt{n}\Theta_{n}(1+i))},&\mbox{if $n$ is odd};\\
    \displaystyle{\frac{\pi}{4}\Theta(i/2)\frac{1}{\sqrt{n}}{\rm Tr}_{H_{\mathfrak{m}}/K}(\sqrt{n}\Theta_{n}(i))},&\mbox{if $n=2m$},
\end{cases}
$$
    where
    $$
    \Theta_{n}(\tau)=\sum_{\substack{p_{1}\cdots p_{\ell}|{\rm rad}(n')\\ p_{1}\cdots p_{\ell}\equiv r\pmod{4}}}\frac{(-1)^{\ell+\frac{r-1}{2}}}{p_{1}\cdots p_{\ell}}\frac{\Theta\left(\frac{n'}{p_{1}\cdots p_{\ell}}\tau/2\right)}{\Theta(\tau/2)}.
    $$
\end{theorem}

\begin{proof}
We only prove the case of $n$ being odd as an illustration. Recall by Proposition~\ref{main1} that one has
\begin{align*}
    &L(E_{n},1)=\frac{\pi}{4}\Theta((1+i)/2)\sum_{\substack{[\mathcal{A}]\in J_{\mathfrak{m}}^{*}/J_{\mathfrak{m},\mathbb{Z}}^{*}\\\mathcal{A}=[a,b+i]\\b\,\,odd}}\frac{\chi'_{n}(\alpha_{\mathcal{A}})\alpha_{\mathcal{A}}}{N(\mathcal{A})}\sum_{\substack{p_{1}\cdots p_{\ell}|{\rm rad}(n')\\ p_{1}\cdots p_{\ell}\equiv r\pmod{4}}}\frac{(-1)^{\ell+\frac{r-1}{2}}}{p_{1}\cdots p_{\ell}}\frac{\Theta\left(\frac{n}{p_{1}\cdots p_{\ell}}\tau_{\mathcal{A}}/2\right)}{\Theta((1+i)/2)}.
    \end{align*}
Then apply Lemmas~\ref{lemchip},~\ref{thetatheta} and~\ref{lemtheta} to further deduce
\begin{align*}    
&L(E_{n},1)
=\frac{\pi\Theta((1+i)/2)}{4\sqrt{n}}\sum_{\substack{[\mathcal{A}]\in J_{\mathfrak{m}}^{*}/J_{\mathfrak{m},\mathbb{Z}}^{*}\\\mathcal{A}=[a,b+i]\\b\,\,odd}}(\sqrt{n})^{\sigma_{\mathcal{A}}^{-1}}\sum_{\substack{p_{1}\cdots p_{\ell}|{\rm rad}(n')\\ p_{1}\cdots p_{\ell}\equiv r\pmod{4}}}\frac{(-1)^{\ell+\frac{r-1}{2}}}{p_{1}\cdots p_{\ell}}\frac{\Theta\left(\frac{n}{p_{1}\cdots p_{\ell}}\tau_{\mathcal{A}}/2\right)}{\Theta(\tau_{\mathcal{A}}/2)}.
\end{align*}
Finally, invoking Lemma~\ref{thetathetagalois} one obtains the desired formula
$$
L(E_{n},1)=
    \displaystyle{\frac{\pi}{4}\Theta((1+i)/2)\frac{1}{\sqrt{n}}{\rm Tr}_{H_{\mathfrak{m}}/K}(\sqrt{n}\Theta_{n}(1+i))},
$$
  where
    $$
    \Theta_{n}(\tau)=\sum_{\substack{p_{1}\cdots p_{\ell}|{\rm rad}(n')\\ p_{1}\cdots p_{\ell}\equiv r\pmod{4}}}\frac{(-1)^{\ell+\frac{r-1}{2}}}{p_{1}\cdots p_{\ell}}\frac{\Theta\left(\frac{n'}{p_{1}\cdots p_{\ell}}\tau/2\right)}{\Theta(\tau/2)}.
    $$

\end{proof}

\section{Factorization lemma}\label{factorlem}

In this section, we assume $n=2^{e}m$ with $e=0$ or~$1$, and $m$ being odd without prime factors $p\equiv3\pmod{4}$. We aim to decompose a CM value of the modular function $\Theta_{n}(\tau)$ defined as in Theorem~\ref{main2} into a sum of CM values of certain modular functions constructed via theta functions of weight~$\frac{1}{2}$. 

The following lemma is due to Rosu \cite{R} via modifying the  factorization lemma of Rodriguez-Villegas and Zagier \cite{V2}.

\begin{lemma}\label{rosulem}
Let $a,D$ be positive integers, and let $\mu,\nu\in\mathbb{Q}$. Then for a complex number~$z$ with ${\rm Im}(z)=y>0$, one has
    $$
\sqrt{\frac{2ay}{D}}    \sum_{r\in\mathbb{Z}/D\mathbb{Z}}\theta_{[a\mu+\frac{ar}{D},\nu]}(Dz/a)\theta_{[\mu+\frac{r}{D},-a\nu]}(-aD\overline{z})=\sum_{m,n}e^{2\pi i(m\nu+nD\mu)}e^{\pi(mni-\frac{|n-mz|^{2}}{2y})\frac{D}{a}},
    $$
    where
    $$
    \theta_{[\mu,\nu]}(\tau)=\sum_{n\in\mathbb{Z}}e^{\pi i(n+\mu)^{2}\tau+2\pi i\nu (n+\mu)}.
    $$
\end{lemma}

\begin{corollary}\label{cordd}
For positive integers $a,a_{1},D$, where $D$ has no prime factors~$p\equiv3\pmod{4}$, take an integer~$b$ such that $b^{2}+1\equiv0\pmod{Da^{2}a_{1}}$. Then the following equality holds:
    $$
    \sqrt{\frac{2}{D^{2}a_{1}}}\sum_{r=0}^{D-1}\theta_{[ar/D,1/2]}(\tau_{\mathcal{A}^{2}\mathcal{A}_{1}})\overline{\theta_{[r/D,1/2]}(\tau_{\mathcal{A}_{1}})}=\sum_{m,n}(-1)^{m}e^{\pi(mni-\frac{|nDaa_{1}-m(b+i)|^{2}}{2Daa_{1}})\frac{D}{a}},
    $$
    where $\tau_{\mathcal{A}}=\frac{b+i}{a}$ denotes the CM point induced by the quadratic integral ideal $[a,b+i]$ of norm~$a$.
\end{corollary}

\begin{proof}
    Take~$D\equiv1\pmod{4}$, $a=N(\mathcal{A})$, $\mu=0$, and $\nu=\frac{1}{2}$. Let $z=\frac{b+i}{Daa_{1}}$ in Lemma~\ref{rosulem}.
\end{proof}

The following  technical lemmas will be useful for relating the above double sum on the right hand
side to a CM value of $\Theta_{n}(\tau)$.

\begin{lemma}\label{lemparity}
    Given an integral ideal $[a,b+i]$ with $a\equiv1\pmod{4}$, and $b$ being odd, one has 
        \begin{align*}
            \{\frac{|na-m(b+i)|^{2}}{a}:\,\, n=2\ell,\,m=2k\}&=\{n^{2}+m^{2}:\,\, n=2\ell,\,m=2k\},\\
             \{\frac{|na-m(b+i)|^{2}}{a}:\,\, n=2\ell+1,\,m=2k\}&=\{n^{2}+m^{2}:\,\, n=2\ell,\,m=2k+1\},\\
              \{\frac{|na-m(b+i)|^{2}}{a}:\,\, n=2\ell,\,m=2k+1\}&=\{n^{2}+m^{2}:\,\, n=2\ell+1,\,m=2k+1\},\\
               \{\frac{|na-m(b+i)|^{2}}{a}:\,\, n=2\ell+1,\,m=2k+1\}&=\{n^{2}+m^{2}:\,\, n=2\ell+1,\,m=2k\}.
        \end{align*}
Similarly, for $b$ being even, one has
 \begin{align*}
            \{\frac{|na-m(b+i)|^{2}}{a}:\,\, n=2\ell,\,m=2k\}&=\{n^{2}+m^{2}:\,\, n=2\ell,\,m=2k\},\\
             \{\frac{|na-m(b+i)|^{2}}{a}:\,\, n=2\ell+1,\,m=2k\}&=\{n^{2}+m^{2}:\,\, n=2\ell,\,m=2k+1\},\\
              \{\frac{|na-m(b+i)|^{2}}{a}:\,\, n=2\ell,\,m=2k+1\}&=\{n^{2}+m^{2}:\,\, n=2\ell+1,\,m=2k\},\\
               \{\frac{|na-m(b+i)|^{2}}{a}:\,\, n=2\ell+1,\,m=2k+1\}&=\{n^{2}+m^{2}:\,\, n=2\ell+1,\,m=2k+1\}.
        \end{align*}
\end{lemma}

\begin{proof}
We only prove the first assertion as an illustration.    Suppose that $b^{2}-ac=-1$. Then
$\frac{|na-m(b+i)|^{2}}{a}$ corresponds to the quadratic form $aX^{2}-2bXY+cY^{2}$. Since $\mathbb{Z}[i]$ is of class number~1, then $aX^{2}-2bXY+cY^{2}$ is equivalent to $x^{2}+y^{2}$ by some $\gamma=\begin{pmatrix}
    A&B\\C&D
\end{pmatrix}\in {\rm SL}_{2}(\mathbb{Z})$ via 
$$
aX^{2}-2bXY+cY^{2}=(AX+BY)^{2}+(CX+DY)^{2},
$$
 which yields that
\begin{align*}
    AD-BC&=1,\\
    A^{2}+C^{2}&=a,\\
    AB+CD&=-b,\\
    B^{2}+D^{2}&=c.
\end{align*}
Since $a$ is odd, then $A$ and $C$ must have different parities. Since $\begin{pmatrix}
    0&-1\\1&0
\end{pmatrix}$
is an isometry of $x^{2}+y^{2}$, and $\begin{pmatrix}
    0&-1\\1&0
\end{pmatrix}\begin{pmatrix}
    A&B\\C&D
\end{pmatrix}=\begin{pmatrix}
    -C&-D\\A&B
\end{pmatrix}$, then upon this, one can just assume $A$ to be even. So, $C$ is odd, and therefore, both $B$ and $D$ are odd. The conclusions  just follow from these together with the relations
$$
x=AX+BY,\quad\mbox{and}\quad y=CX+DY
$$
and simple analysis for their parities according to that of $X$ and $Y$.
\end{proof}

\begin{lemma}\label{lemaab}
    Let  $\mathcal{A}=(\alpha_{\mathcal{A}})=[a,b_{\mathcal{A}}+i]$ and $\mathcal{A}=(\alpha_{\mathcal{A}_{1}})=[a_{1},b_{\mathcal{A}_{1}}+i]$ be two primitive integral ideals with $a,a_{1}\equiv1\pmod{4}$. Then there is a number $b$ such that $b^{2}\equiv-1\pmod{Da^{2}a_{1}}$ for a given number $D$ that has no prime factors~$p\equiv3\pmod{4}$. In this case, write
    $$
    \mathcal{A}=[a,b+i]\quad{and}\quad \mathcal{A}_{1}=[a_{1},b+i].
    $$
\end{lemma}

\begin{proof}
    Since both $\mathcal{A}$ and $\mathcal{A}_{1}$ are primitive, then both $a$ and $a_{1}$ have prime factors that are congruent to~1 modulo~4 only. The conclusion therefore follows from the Chinese Remainder Theorem.
\end{proof}

\begin{lemma}\label{lemtheta1}
Define
$$
{\Theta}_{o}(\tau)=\sum_{m,n\,odd}q^{m^{2}+n^{2}}.
$$
Then
    \begin{align*}
  \Theta(\tau)+ {\Theta}_{o}(\tau/4)&=\Theta(\tau/2).
    \end{align*}
\end{lemma}

\begin{proof}
    This result is classical. See, e.g., \cite[Chapter~3]{C17}.
\end{proof}

\begin{lemma}\label{parity}
Let $b$ be an odd number given by Lemma~\ref{lemaab}. Then
 
     \begin{align}
        \sum_{m,n\,\,even}(-1)^{m}e^{\pi(mni-\frac{|nDaa_{1}-m(b+i)|^{2}}{2Daa_{1}})\frac{D}{a}}&=\sum_{k,\ell}e^{2\pi i\frac{|\ell Daa_{1}-k(b+i)|^{2}}{Daa_{1}}D\frac{b+i}{a}}=\Theta(D\tau_{\mathcal{A}}),\label{odd1}\\
        \sum_{\substack{m\,even\\n\,odd}}(-1)^{m}e^{\pi(mni-\frac{|nDaa_{1}-m(b+i)|^{2}}{2Daa_{1}})\frac{p}{a}}&=e^{-\pi i\frac{b}{2}}\sum_{\substack{m\,even\\n\,odd}}e^{2\pi i\frac{|nDaa_{1}-m(b+i)|^{2}}{Daa_{1}}D\frac{b+i}{4a}},\\
   \sum_{\substack{m\,odd\\n\,even}}(-1)^{m}e^{\pi(mni-\frac{|nDaa_{1}-m(b+i)|^{2}}{2Daa_{1}})\frac{D}{a}}&=\sum_{\substack{m\,odd\\n\,even}}e^{2\pi i\frac{|nDaa_{1}-m(b+i)|^{2}}{Daa_{1}}D\frac{b+i}{4a}}=\Theta_{o}(D\tau_{\mathcal{A}}/4),\label{odd2}\\
        \sum_{m,n\,\,odd}(-1)^{m}e^{\pi(mni-\frac{|npaa_{1}-m(b+i)|^{2}}{2Daa_{1}})\frac{D}{a}}&=-e^{-\pi i\frac{b}{2}}\sum_{m,n\,odd}e^{2\pi i\frac{|nDaa_{1}-m(b+i)|^{2}}{Daa_{1}}D\frac{b+i}{4a}}.
  \end{align}

   Similarly, let $b$ be an even number given by Lemma~\ref{lemaab}. Then
 \begin{align}
        \sum_{m,n\,\,even}(-1)^{m}e^{\pi(mni-\frac{|nDaa_{1}-m(b+i)|^{2}}{2Daa_{1}})\frac{D}{a}}&=\sum_{k,\ell}e^{2\pi i\frac{|\ell Daa_{1}-k(b+i)|^{2}}{Daa_{1}}D\frac{b+i}{a}}=\Theta(D\tau_{\mathcal{A}}),\label{even1}\\
        \sum_{\substack{m\,even\\n\,odd}}(-1)^{m}e^{\pi(mni-\frac{|nDaa_{1}-m(b+i)|^{2}}{2Daa_{1}})\frac{p}{a}}&=e^{-\pi i\frac{b}{2}}\sum_{\substack{m\,even\\n\,odd}}e^{2\pi i\frac{|nDaa_{1}-m(b+i)|^{2}}{Daa_{1}}D\frac{b+i}{4a}},\\
        \sum_{\substack{m\,odd\\n\,even}}(-1)^{m}e^{\pi(mni-\frac{|nDaa_{1}-m(b+i)|^{2}}{2Daa_{1}})\frac{D}{a}}&=-e^{-\pi i\frac{b}{2}}\sum_{\substack{m\,odd\\n\,even}}e^{2\pi i\frac{|nDaa_{1}-m(b+i)|^{2}}{Daa_{1}}D\frac{b+i}{4a}},\\
        \sum_{m,n\,\,odd}(-1)^{m}e^{\pi(mni-\frac{|npaa_{1}-m(b+i)|^{2}}{2Daa_{1}})\frac{D}{a}}&=\sum_{m,n\,odd}e^{2\pi i\frac{|nDaa_{1}-m(b+i)|^{2}}{Daa_{1}}D\frac{b+i}{4a}}=\Theta_{o}(D\tau_{\mathcal{A}}/4).\label{even2}
    \end{align}
\end{lemma}

\begin{proof}
The verification of the first equality in each part is straightforward; for example, for the first equality of~\eqref{odd1}, note that
 \begin{align*}
        &\pi(mni-\frac{|nDaa_{1}-m(b+i)|^{2}}{2Daa_{1}})\frac{D}{a}-2\pi i\frac{|\ell Daa_{1}-k(b+i)|^{2}}{Daa_{1}}D\frac{b+i}{a}\\
        &=\pi(4k\ell i-\frac{|2\ell Daa_{1}-2k(b+i)|^{2}}{2Daa_{1}})\frac{D}{a}-2\pi i\frac{|\ell Daa_{1}-k(b+i)|^{2}}{Daa_{1}}D\frac{b+i}{a}\\
        &=2\pi i\left(\frac{2k\ell D}{a} +\frac{|\ell Daa_{1}-k(b+i)|^{2}}{a^{2}a_{1}}i-\frac{|\ell Daa_{1}-k(b+i)|^{2}}{a^{2}a_{1}}b-\frac{|\ell Daa_{1}-k(b+i)|^{2}}{a^{2}a_{1}}i\right)\\
        &=2\pi i\left(\frac{2k\ell D(b^{2}+1)}{a}-\ell^{2}D^{2}a_{1}b-\frac{k^{2}(b^{2}+1)b}{a^{2}a_{1}}\right)\\
        &=0\pmod{2\pi i\mathbb{Z}}.
    \end{align*}
We leave the corresponding details of the others to the reader. Finally, the second equalities of~\eqref{odd1},~\eqref{odd2},~\eqref{even1} and~\eqref{even2} follow from Lemmas~\ref{lemparity} and~\ref{lemtheta1}.

    
   
\end{proof}

As a consequence of  the preceding lemmas, one can decompose the CM value $\Theta(D\tau_{\mathcal{A}}/2)$ in terms of CM values of $\theta_{[r/D,1/2]}(\tau)$ as follows.

\begin{corollary}\label{cortheta}
    Let $D$ be a positive integer with no prime factors~$p\equiv3\pmod{4}$, and let 
    $$
    \mathcal{A}=[a,b+i]\quad\mbox{and}\quad \mathcal{A}_{1}=[a_{1},b+i]
    $$
    be primitive integral ideals with $a,a_{1}\equiv1\pmod{4}$, and $b$  such that $b^{2}\equiv-1\pmod{Da^{2}a_{1}}$. Then one has 
    $$
     \sqrt{\frac{2}{D^{2}a_{1}}}\sum_{r=0}^{D-1}\theta_{[ar/D,1/2]}(\tau_{\mathcal{A}^{2}\mathcal{A}_{1}})\overline{\theta_{[r/D,1/2]}(\tau_{\mathcal{A}_{1}})}=\Theta(D\tau_{\mathcal{A}}/2).
    $$
    In particular, when $D=1$, 
    $$
    \sqrt{\frac{2}{a_{1}}}\theta_{[0,1/2]}(\tau_{\mathcal{A}^{2}\mathcal{A}_{1}})\overline{\theta_{[0,1/2]}(\tau_{\mathcal{A}_{1}})}=\Theta(\tau_{\mathcal{A}}/2).
    $$
\end{corollary}

\begin{proof}
    Upon Corollary~\ref{cordd}, separate the double sum on the right hand side wherein according to the parities of $m,n$. Making use of Lemmas~\ref{lemtheta1} and~\ref{parity}, one obtains the desired formula.
\end{proof}


\begin{proposition}\label{corthetaf}
      Write 
$$
n'=\begin{cases}
    n,&\mbox{if $n$ is odd};\\
    m,&\mbox{if $n=2m$}.
\end{cases}
$$   
and denote by ${\rm rad}(n')$ the radical of~$n$, and let 
    $$
    \mathcal{A}=[a,b+i]\quad\mbox{and}\quad \mathcal{A}_{1}=[a_{1},b+i]
    $$
    be two primitive integral ideals with $b$ being  such that $b^{2}\equiv-1\pmod{n'a^{2}a_{1}}$. Then
    \begin{align*}
       \Theta_{n}(\tau_{\mathcal{A}})&=\frac{1}{n'}\sum_{p_{1}\cdots p_{\ell}|{\rm rad}(n')}(-1)^{\ell}+\frac{1}{n'}\sum_{r\in(\mathbb{Z}/n'\mathbb{Z})^{\times}}f_{ar,n'}(\tau_{\mathcal{A}^{2}\mathcal{A}_{1}})\overline{f_{r,n'}(\tau_{\mathcal{A}_{1}})},
    \end{align*}
    where $\Theta_{n}(\tau)$ is defined as in Theorem~\ref{main2}, and
    $$
    f_{r,d}(\tau)=\frac{\theta_{[r/d,1/2]}(\tau)}{\theta_{[0,1/2]}(\tau)}.
    $$
\end{proposition}

\begin{proof}
By Corollary~\ref{cortheta},  one has for a positive integer $D|n'$ that
$$
\frac{\Theta(D\tau_{\mathcal{A}}/2)}{\Theta(\tau_{\mathcal{A}}/2)}=\frac{1}{D}\sum_{r=0}^{D-1}f_{ar,D}(\tau_{\mathcal{A}^{2}\mathcal{A}_{1}})f_{r,D}(\tau_{\mathcal{A}_{1}}).
$$
So it is routine to deduce that
     \begin{align*}
       \Theta_{n}(\tau_{\mathcal{A}})&= \sum_{p_{1}\cdots p_{\ell}|{\rm rad}(n')}\frac{(-1)^{\ell}}{p_{1}\cdots p_{\ell}}\frac{\Theta\left(\frac{n'}{p_{1}\cdots p_{\ell}}\tau_{\mathcal{A}}/2\right)}{\Theta(\tau_{\mathcal{A}}/2)}\\
       &= \frac{1}{n'}\sum_{p_{1}\cdots p_{\ell}|{\rm rad}(n')}(-1)^{\ell}\sum_{r=0}^{\frac{n'}{p_{1}\cdots p_{\ell}}-1}f_{ar,\frac{n'}{p_{1}\cdots p_{\ell}}}(\tau_{\mathcal{A}^{2}\mathcal{A}_{1}})\overline{f_{r,\frac{n'}{p_{1}\cdots p_{\ell}}}(\tau_{\mathcal{A}_{1}})}\\
       &=\frac{1}{n'}\sum_{p_{1}\cdots p_{\ell}|{\rm rad}(n')}(-1)^{\ell}\sum_{r=0}^{\frac{n'}{p_{1}\cdots p_{\ell}}-1}f_{a(p_{1}\cdots p_{\ell}r),n'}(\tau_{\mathcal{A}^{2}\mathcal{A}_{1}})\overline{f_{(p_{1}\cdots p_{\ell}r),n'}(\tau_{\mathcal{A}_{1}})}\\
       &=\frac{1}{n'}\sum_{p_{1}\cdots p_{\ell}|{\rm rad}(n')}(-1)^{\ell}+\frac{1}{n'}\sum_{r\in(\mathbb{Z}/n'\mathbb{Z})^{\times}}f_{ar,n'}(\tau_{\mathcal{A}^{2}\mathcal{A}_{1}})\overline{f_{r,n'}(\tau_{\mathcal{A}_{1}})}
    \end{align*}
    as desired.
\end{proof}




\section{$L(E_{n},1)$ as a square}\label{square}

This section is devoted to expressing $L(E_{n},1)$ as a square that justifies the main result of the present work. To this end, besides the preceding preliminaries, an additional number of technicalities are needed. As  in Section~\ref{factorlem}, we assume~$n=2^{e}m$ with $e=0$ or~$1$, and $m$ being odd with no prime factors~$p\equiv3\pmod{4}$.



\begin{lemma}\label{lemna}
Given a positive odd integer~$n$.    For a class $[\mathcal{A}]\in J_{\mathfrak{m}}^{*}/J_{\mathfrak{m},\mathbb{Z}}^{*}$,  there is a representative $\mathcal{A}$ such that $\mathcal{A}=\alpha_{\mathcal{A}}\mathcal{O}_{K}=[a,b+i]$, and $n_{a}\equiv1\pmod{n}$, where $n_{a}$ is defined by
$$
n_{a}a+m_{a}(b+i)=\alpha_{\mathcal{A}}.
$$

\end{lemma}

\begin{proof}
    Suppose that $\tilde{\mathcal{A}}={\alpha}_{\tilde{\mathcal{A}}}\mathcal{O}_{K}$ is of norm~$\tilde{a}$ and is primitive. Pick a primitive integral ideal $\lambda\mathcal{O}_{K}\in J_{\mathfrak{m},\mathbb{Z}}^{*}$ such that $[A,B+i]=\lambda\equiv \ell\pmod{n}$ with $\gcd(\ell,n)=1$ and $\ell\equiv 1\pmod{4}$. Upon this, one can find by the Chinese Remainder Theorem an odd number $b$ such that
 $$
    \alpha_{\tilde{\mathcal{A}}}\mathcal{O}_{K}=[\tilde{a},b+i],\quad\mbox{and}\quad \lambda\mathcal{O}_{K}=[A,b+i],
    $$
    and 
     $$
     \alpha_{\tilde{\mathcal{A}}}=n_{1}\tilde{a}+m_{1}(b+i)
     $$
     with $\gcd(n_{1},n)=1$.
     So, one has
     $$
     \lambda=n_{2}A+m_{2}(b+i)
     $$
      for some $n_{2},m_{2}$ with $n|m_{2}$, and
    $$
    \lambda\alpha_{\tilde{\mathcal{A}}}\mathcal{O}_{K}=[A\tilde{a},b+i],
    $$
    which is still primitive. Thus, 
    $$
    \lambda\alpha_{\tilde{\mathcal{A}}}=n_{1}n_{2}A\tilde{a}+n_{2}Am_{1}(b+i)\pmod{n}.
    $$
    On the other hand, one can write
    $$
    \lambda\alpha_{\tilde{\mathcal{A}}}=NA\tilde{a}+M(b+i)
    $$
    for some $N,M$.
Comparing this with its last line,    it is not hard to see that
    $$
    NA\tilde{a}\equiv n_{1}n_{2}A\tilde{a}\pmod{n},
    $$
    that is,
    $$
    N\equiv n_{1}n_{2}\pmod{n}.
    $$
    Note that $A\equiv \ell^{2}\pmod{n}$, and that $\lambda\equiv\ell\pmod{n}$ implies that $n_{2}A\equiv\ell\pmod{n}$. Thus, $n_{2}\equiv\ell^{-1}\pmod{n}$, which together with the line above implies that
    $$
    N\equiv n_{1}\ell^{-1}\pmod{n}.
    $$
  Finally, take $\ell\equiv n_{1}\pmod{n}$, so that $\mathcal{A}=\lambda\alpha_{\tilde{\mathcal{A}}}\mathcal{O}_{K}$ is the desired representative.
\end{proof}


\begin{lemma}\label{shimuraff}
For a positive odd integer~$D|n$, let $\alpha\mathcal{O}_{K}=[a,b+i]=[a,a\tau_{\mathcal{A}}]$ be a primitive integral ideal with $\alpha\equiv1\pmod{4}$,  $a\equiv1\pmod{4}$, $b$  such that $b^{2}\equiv-1\pmod{D^{2}a}$, and $n_{a}\equiv1\pmod{4D}$ given by Lemma~\ref{lemna}. Then one has
$$
f_{r,D}(\tau_{\mathcal{A}})=f_{r,D}(b+i)
$$
for any~$r$.
\end{lemma}

\begin{proof}
    This is basically a consequence of the Shimura reciprocity law (see, e.g., \cite[Appendix]{R}) and the transformation formula for $\theta_{[\mu,\nu]}(\tau)$ (see, e.g. \cite{FK}). We omit the details.
\end{proof}


Building upon the preceding preparations,  we are ready to  express $L(E_{n},1)$ as a square, which is the main result of the present work.

\begin{theorem}\label{main4}
\begin{enumerate}
    \item For~$n$ being odd, and an even number~$b$ such that $b^{2}=-1\pmod{n^{2}}$, one has
    \[
L(E_n,1)=\frac{\pi \left|\theta_{\chi_{n}}\left(\frac{b+i}{2
n^2}\right)\right|^2}{4\sqrt{2}n }.
\]

\item For~$n=2m$, and an odd number~$b$ such that $b^{2}=-1\pmod{m^{2}}$, one has
    \[
L(E_n,1)=\frac{\pi \left|\theta_{\chi_{m}}\left(\frac{b+i}{2
m^2}\right)\right|^2}{\sqrt{2}n }.
\]
\end{enumerate}

\end{theorem}

\begin{proof}

We only prove part (1) as an illustration.  Note that the proof of Theorem~\ref{main2} together with Lemma~\ref{lemna} yields that
     $$
   L(E_{n},1)=\frac{\pi\Theta((1+i)/2)}{4}\sum_{\substack{[\mathcal{A}]\in J_{\mathfrak{m}}^{*}/J_{\mathfrak{m},\mathbb{Z}}^{*}\\ \mathcal{A}=[a,b+i]\\b\,\,odd\\b^{2}\equiv-1\pmod{n^{2}a^{2}}\\n_{a}\equiv1\pmod{4n}}}\chi_{n}'(\mathcal{A})\Theta_{n}(\tau_{\mathcal{A}}).
    $$
Upon this,    using Proposition~\ref{corthetaf} and the simple fact of
$$
\sum_{[\mathcal{A}]\in J_{\mathfrak{m}}^{*}/J_{\mathfrak{m},\mathbb{Z}}^{*}}\chi'_{n}(\mathcal{A})=0,
$$
as well as Lemma~\ref{shimuraff},  one can deduce from the equality above that
   \begin{align*}
        L(E_{n},1)&=\frac{\pi\Theta((1+i)/2)}{4n}\sum_{\substack{[\mathcal{A}]\in J_{\mathfrak{m}}^{*}/J_{\mathfrak{m},\mathbb{Z}}^{*}\\ \mathcal{A}=[a,b+i]\\b\,\,odd\\b^{2}\equiv-1\pmod{n^{2}a^{2}}\\n_{a}\equiv1\pmod{4n}}}\sum_{r\in(\mathbb{Z}/n\mathbb{Z})^{\times}}\chi_{n}'(\mathcal{A})f_{ar,n}(b+i)\overline{f_{r,n}(b+i)}\\
         &=\frac{\pi\Theta((1+i)/2)}{4n}\sum_{a\in(\mathbb{Z}/n\mathbb{Z})^{\times}}\sum_{r\in(\mathbb{Z}/n\mathbb{Z})^{\times}}\chi_{n}(ar)f_{ar,n}(b+i)\overline{\chi_{n}(r)f_{r,n}(b+i)}\\
         &=\frac{\pi\Theta((1+i)/2)}{4n}\left|\sum_{r=1}^{n-1}\chi_{n}(r)f_{r,n}(\tilde{b}+i)\right|^{2}
    \end{align*}
    for a $\tilde{b}$ being odd such that $\tilde{b}^{2}+1=0\pmod{n^{2}}$. Finally, by the definition of $f_{r,n}(\tau)$, it is routine to check that
    $$
    \sum_{r=1}^{n-1}\chi_{n}(r)f_{r,n}(\tilde{b}+i)=\frac{\theta_{\chi_{n}}\left(\frac{b+i}{2n^{2}}\right)}{\theta(i/2)}
    $$
    for a $b$ being even such that $b^{2}+1=0\pmod{n^{2}}$. Finally, noting that
    $$
    \Theta\left(\frac{1+i}{2}\right)=\theta\left(\frac{1+i}{2}\right)^{2}
    $$
    and making use of the classical result
    $$
    \frac{\theta(i/2)^{2}}{\theta((1+i)/2)^{2}}=\sqrt{2}
    $$
    give the desired formula.

\end{proof}

\begin{remark}\label{rem1}
    Following the development of the formula given in part (2) of Theorem~\ref{main4}, one may also establish the alternative formulation
    $L(E_{n},1)={\left|\theta_{\chi_{2n}}\left(\tau_{n/2}/4\right)\right|}, $
    where $\tau_{n/2}=\frac{b+i}{2m^{2}}$ for some even integer~$b$ such that $b^{2}\equiv-1\pmod{m^{2}}$. This basically follows from taking $\mu=\frac{1}{2}$ instead of~$0$ in Lemma~\ref{rosulem} and accordingly modifying all the details afterwards. We leave the details to the reader.
\end{remark}

\section{Mock Heegner zeros of theta functions}\label{mock}
Let $N  $ be a positive integer, and  $f$  a nonzero meromorphic modular form of level $ N $ with algebraic Fourier coefficients. The zeros and poles of $ f $ define an algebraic divisor on $ X_0(N)$. Such divisors have been studied extensively in connection with algebraic and analytic aspects of modular forms. 

We consider the zeros and   poles of $f$ on the upper half plane. 
 It is well known that such points are either CM (complex multiplication) points  or transcendental. In the case of CM points  of discriminant $D$, if every prime divisor of $N$  splits in $\mathbb{Q}(\sqrt{D})$, then these points are usually  referred to as Heegner points. The hypothesis that every prime divisor of $N$  splits in $\mathbb{Q}(\sqrt{D})$ is usually called the Heegner hypothesis.  
The zeros of modular forms that occur at Heegner points are very important in number theory.  
  Yang \cite{Y2.5}  constructed a family of infinite many theta series over the Hilbert-Blumenthal modular surfaces with a common zero by the method in \cite{V4}. Later  Jimenez-Urroz and Yang \cite{Y1} studied the  Heegner zeros of the theta functions for the family of CM elliptic curves constructed by B. Gross.

For $E_n: y^2=x^3-n^2x$, the conductor
$N=32n^2$ if $n$ is odd and $16n^2$ if $n$ is even. The CM points we considered in this paper are all in $\mathbb{Q}(i)$. Therefore,  the Heegner hypothesis is not satisfied. Following Monsky's notation \cite{Monsky}, we will call such these CM points mock Heegner zeros if the theta function vanishes at them. 

 By  Theorem \ref{Theorem1.1}, 
$\tau_n$ is a  zero of $\theta_{\chi_n}$  if and only if $L(E_n,1)=0$. Hence Theorem \ref{Theorem1.1} can give a lot of information on  zeros of $\theta_{\chi_n}$.

\begin{theorem}\label{Theorem8.3}
Let  $m=\prod_{i=1}^s p_i$ with  each prime $p_i\equiv1\pmod{4}$, $b$  an even integer such that $b^2\equiv -1\pmod{m^2}$, $\tau_m=\frac{b+i}{2
m^2}$, $\tau'_{m}=\frac{b+m^2+i}{2
m^2}$.  
\begin{enumerate}
    \item If $m\equiv 5\pmod{8}$, then $\tau_m$ and  $-\overline{\tau}_m$ are  zeros of $\theta_{\chi_m}$. 
    \item If  $m\equiv 1\pmod{8}$, then $\tau_m$ and  $-\overline{\tau}_m$ are  zeros of $\theta_{\chi_m}$ if and only if $L(E_m,1)$=0.
    \item If $L(E_{2m},1)=0$, then $\tau'_{m}$, $-\overline{\tau'}_{m}$ are  zeros of $\theta_{\chi_m}$ and $\tau_m/4$,  $-\overline{\tau'}_{m}/4$ are  zeros of $\theta_{\chi_{4m}}$.
\end{enumerate}
\end{theorem}
\begin{proof}
All of these assertions are directly proved by Theorem \ref{Theorem1.1}.
\end{proof}

\begin{theorem}\label{Theorem8.0}
Let $n=\prod_{k=1}^t p_k$ such that  each $p_k\equiv 1\pmod 4$,  $b$ an even integer satisfying $b^2+1=\lambda {n^2}$ for some $\lambda\in \mathbb{Z}$.  Let $\sigma=\begin{pmatrix}
2b&-\lambda\\
4n^2&-2b
\end{pmatrix}
$ be an Atkin-Lehner involution. Then we have 
\[
\theta_{\chi_n}(\sigma(\tau))=\begin{cases}
(1-i)\sqrt{n^2\tau-{b}/{2}} \theta_{\chi_n}(\tau),&\ \text{if }\ 
n\equiv 1\pmod 8;  \\
(i-1)\sqrt{n^2\tau-{b}/{2}}\theta_{\chi_n}(\tau),&\ \text{if }\ 
n\equiv 5\pmod 8,
\end{cases}
\]
where the square root takes the principal branch of the square root function. 
\end{theorem}
\begin{proof}Let  $\gamma=\begin{pmatrix}
b&-\lambda n\\
n&-b
\end{pmatrix}
$.  Then $\sigma(\tau)=\gamma({2n\tau})/2n$. 
Let $$\theta_{n,h}(\tau)=\sum_{m\in\mathbb Z} e^{2\pi i(mn+h)^2\tau}, \quad h=1,2,\cdots, {n-1}.$$
Then $\theta_{\chi_n}(\tau)=\sum_{h=1}^{{n-1}}\chi(h)\theta_{n,h}(\tau)$.  By Proposition 2.2 of \cite{B}, the action of $\gamma$  sends the subscript $h$ in the theta function $\theta_{n,h}$ to $bh$. When $n \equiv 1 \pmod{8}$, we have $\chi_n(b) = 1$; whereas when $n \equiv 5 \pmod{8}$, we have $\chi_n(b) = -1$.
By carefully computing the coefficients, one can prove
the theorem. 
\end{proof}

This theorem was first pointed out, along with an outline of its proof, by Noam Elkies in his MathOverflow answer to the question “CM zeros of unary theta series” posed by Henri Cohen \cite{Cohen}.  However, the subtle distinction between the cases $n \equiv 5\pmod 8$ and $n \equiv 1\pmod 8$  appears to have gone unnoticed.

\begin{corollary}\label{Corollary8.0}
Let $n=\prod_{k=1}^t p_k$ such that  each $p_k\equiv 1\pmod 4$,  $\tau_n=\frac{b+i}{2n^2}$, where $b$ is an even integer satisfying $b^2+1=\lambda {n^2}$ for some $\lambda\in \mathbb{Z}$.  
Then $\tau_n$ is a zero of $\theta_{\chi_n}(\tau)$ if $n\equiv 5\pmod 8 $.  If $n\equiv 1\pmod 8 $, then the vanishing order 
of  $\theta_{\chi_n}(\tau)$  at $\tau_n$ is $0$ or  at least $2$.  
\end{corollary}
\begin{proof}
Note that $\sigma(\tau_n)=\tau_n$. By Theorem ~\ref{Theorem8.0}, we have 
\[
\theta_{\chi_n}(\tau_n)=\begin{cases}
 \theta_{\chi_n}(\tau_n),&\ \text{if }\ 
n\equiv 1\pmod 8;  \\
-\theta_{\chi_n}(\tau_n),&\ \text{if }\ 
n\equiv 5\pmod 8.
\end{cases}
\]
Hence we have  $\theta_{\chi_n}(\tau_n)=0$ if $n\equiv 5\pmod 8 $.
Assume that  we have  the Taylor expansion $\theta_{\chi_n}(\tau)=a_0+a_1(\tau-\tau_n)+\cdots$. Then 
$\theta_{\chi_n}(\sigma(\tau))=a_0-{a_1}(\tau-\tau_n)+\cdots$ and $ (1-i)\sqrt{n^2\tau-{b}/{2}} \theta_{\chi_n}(\tau) =a_0+(a_1-in^2a_0)(\tau-\tau_n)+\cdots$. 
 If $n\equiv 1\pmod 8 $ and $a_0=0$, then we have $a_1=-a_1$
  which implies that $a_1=0$, i.e., 
  the vanishing order 
of  $\theta_{\chi_n}(\tau)$  at $\tau_n$ is  at least $2$.  
\end{proof}

We also did   numerical computation on the order of $\theta_{\chi_n}(\tau)$ vanishing at $\tau_n$.  It seems that  $\tau_n$ is always a simple zero of $\theta_{\chi_n}$ if $n\equiv 5\pmod 8 $, and a double zero if $n\equiv 1\pmod 8 $ and $n$ is congruent number.    We have verified that for  all prime number $n<10^4$ such that $n\equiv 1\pmod 8$ and $n$ is a congruent number and found that $\tau_n$ is always a double zero of $\theta_{\chi_n}$.

By applying Theorem \ref{Theorem1.1}, we obtain an effective method to compute the order of the Tate–Shafarevich group and locate the zeros of $\theta_{\chi_n}$ via the following proposition.
\begin{proposition}\label{prop8.4}
    Let $n$ and $\tau_n$ be as in Theorem \ref{Theorem1.1}.  
If $L(E_n,1)\neq 0$, then 
\[
{\#\Sha(E_n)}=\begin{cases}\displaystyle
\frac{\pi|\theta_{\chi_{n}}\left(\tau_n\right)|^2}{\sqrt{2n}\varpi\sigma_0^2(n)},&\ \text{if}\ n \ \text{is odd};\\
\displaystyle\frac{\sqrt{2}\pi\left|\theta_{\chi_{n/2}}\left(\tau'_{n/2}\right)\right|^2}{\sqrt{n}\varpi\sigma_0^2(n/2)},&\ \text{if}\ n \ \text{is even},
\end{cases}
\]
where $\sigma_0(n)$ denotes the number of divisors of $n$. In particular, if $L(E_n,1)\neq 0$, then  
\[
\begin{aligned}
    |\theta_{\chi_{n}}\left(\tau_n\right)|&\geq {{(2n)^{\frac{1}{4 }}}\varpi^{\frac{1}{2}}\pi^{-\frac{1}{2}}}\sigma_0(n),& \ &\text{if}\ n\ \text{is odd};\\
|\theta_{\chi_{n}}\left(\tau'_{n/2}\right)|&\geq {{(n/2)^{\frac{1}{4 }}}\varpi^{\frac{1}{2}}\pi^{-\frac{1}{2}}}\sigma_0(n/2),& \ &\text{if}\ n\ \text{is even}. 
\end{aligned}
\]
\end{proposition}
\begin{proof}
This follows directly from  Theorem \ref{Theorem1.1} and Theorem 3 of \cite{Tunnell}.
\end{proof}
Therefore, if for some odd \(n\) one computes $$|\theta_{\chi_{n}}\left(\tau_n\right)|<{{(2n)^{\frac{1}{4 }}}\varpi^{\frac{1}{2}}\pi^{-\frac{1}{2}}}\sigma_0(n),$$  then necessarily $L(E_n,1)=0$. This provides an effective and computationally accessible criterion for finding congruent numbers and  mock Heegner zeros of theta functions.

\end{document}